\documentclass[12pt]{article}
\usepackage{e-jc}
\specs{P1.60}{29(1)}{2022}{10274}

\usepackage{amsmath}
\usepackage{amssymb, graphicx, enumerate, enumitem, color}
\usepackage[usenames,dvipsnames,svgnames,table]{xcolor}

\usepackage{multirow}
\usepackage{longtable}
\usepackage{booktabs}

\dateline{Mar 1, 2021}{Jan 10, 2022}{Mar 25, 2022}
\MSC{05C15, 05C83}
\Copyright{The authors. Released under the CC BY-ND license (International 4.0).}

\colorlet{mylinkcolor}{violet}
\colorlet{mycitecolor}{YellowOrange}
\colorlet{myurlcolor}{Aquamarine}
\usepackage{hyperref}
\usepackage{xcolor}
\hypersetup{
    pdfauthor={Gwena\"el Joret, Piotr Micek},
    colorlinks,
    linkcolor={red!50!black},
    citecolor={blue!50!black},
    urlcolor={blue!80!black}
}

\usepackage{cleveref}

\newcommand\restr[2]{{
  \left.\kern-\nulldelimiterspace 
  #1 
  \vphantom{\big|} 
  \right|_{#2} 
  }}

\newcommand{\set}[1]{\{#1\}}

\newcommand{\calB}{\mathcal{B}}
\newcommand{\calC}{\mathcal{C}}

\newcommand{\Oh}{\mathcal{O}} 

\newcommand{\NN}{\mathbb{N}}

\DeclareMathOperator\tw{tw}
\DeclareMathOperator\stw{stw}

\DeclareMathOperator\td{td}
\DeclareMathOperator\clos{clos}

\DeclareMathOperator\wcol{wcol}
\DeclareMathOperator\scol{scol}
\DeclareMathOperator\col{col}

\DeclareMathOperator\dist{dist}
\DeclareMathOperator\WReach{WReach}

\DeclareMathOperator\dom{dom}

\let\leq\leqslant
\let\geq\geqslant

\let\subset\subseteq

\let\epsilon\varepsilon

\title{Improved bounds for weak coloring numbers}

\author{Gwena\"el Joret\thanks{Supported by an ARC grant from the Wallonia-Brussels Federation of Belgium and a CDR grant from the National Fund for Scientific Research (FNRS).}\\
\small Computer Science Department\\[-0.8ex]
\small Universit\'e libre de Bruxelles\\[-0.8ex] 
\small Brussels, Belgium\\
\small\tt gjoret@ulb.ac.be\\
\and
Piotr Micek\thanks{Supported by the National Science Center of Poland under grant no.\ 2018/31/G/ST1/03718.}\\
\small Theoretical Computer Science Department \\[-0.8ex]
\small Jagiellonian University \\[-0.8ex] 
\small Krak\'ow, Poland\\
\small\tt piotr.micek@uj.edu.pl\\
}

\begin{document}

\maketitle

\begin{abstract}
Weak coloring numbers generalize the notion of degeneracy of a graph. They were introduced by Kierstead \& Yang in the context of games on graphs. Recently, several connections have been uncovered between weak coloring numbers and various parameters studied in graph minor theory and its generalizations. In this note, we show that for every fixed $k\geq1$, the maximum $r$-th weak coloring number of a graph with simple treewidth $k$ is $\Theta(r^{k-1}\log r)$. As a corollary, we improve the lower bound on the maximum $r$-th weak coloring number of planar graphs from $\Omega(r^2)$ to $\Omega(r^2\log r)$, and we obtain a tight bound of $\Theta(r\log r)$ for outerplanar graphs.
\end{abstract}

\maketitle

\section{Introduction}
All graphs in this paper are finite, simple, and undirected. 
The \emph{coloring number} $\col(G)$ of a graph $G$ is the least number $k$ such that $G$ has a vertex ordering in which each vertex is preceded by at most $k-1$ of its neighbors. 
If we color the vertices of $G$ one by one following such a vertex ordering, assigning to each vertex the smallest color not already used on its neighbors, then we use at most $\col(G)$ colors.  
This proves the first inequality in the following observation:
\[
\chi(G) \leq \col(G) \leq \Delta(G)+1,
\]
where $\chi(G)$ denotes the chromatic number of $G$ and $\Delta(G)$ denotes the maximum degree of $G$. 
The focus of this paper is a family of graph parameters, the {\em weak coloring numbers}, which can be thought of as relaxations of the coloring number.

The \emph{length} of a path is the number of its edges.
For two vertices $u$ and $v$ in a graph $G$, an $u$--$v$ \emph{path} is a path in $G$ with ends in $u$ and $v$. 
Let $G$ be a graph and let $\sigma$ be an ordering of the vertices of $G$. 
For $r\in\set{0,1,2,\ldots}\cup\set{\infty}$ and two vertices $u$ and $v$ of $G$, we say that
$u$ is \emph{weakly $r$-reachable} from $v$ in $\sigma$, if there exists
an $u$--$v$ path of length at most $r$ 
such that for every vertex $w$ on the path, $u\leq_{\sigma} w$.
See Figure~\ref{fig:wcol-def}.
The set of vertices that are weakly $r$-reachable from a vertex $v$ in $\sigma$ is denoted by $\WReach_r[G, \sigma, v]$. 
We define
\begin{align*}
\wcol_r(G, \sigma) &= \max_{v \in V(G)}\ |\WReach_r[G, \sigma, v]|,\\
\wcol_r(G) &= \min_{\sigma}\ \wcol_r(G, \sigma),
\end{align*}
where $\sigma$ ranges over the set of all vertex orderings of $G$.
We call $\wcol_r(G)$ the $r$-\emph{th weak coloring number} of $G$.
Clearly, 
\[
\col(G) = \wcol_1(G) \leq \wcol_2(G) \leq \cdots\leq \wcol_{\infty}(G).
\]
\begin{figure}[!h]
    \centering
    \includegraphics{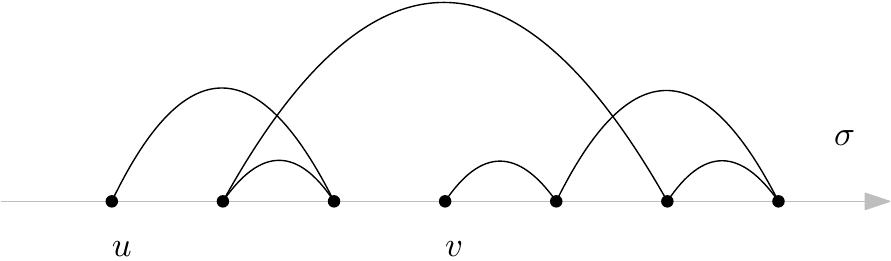}
    \caption{$u$ is $6$-weakly reachable from $v$ in $\sigma$.}
    \label{fig:wcol-def}
\end{figure}

Weak coloring numbers were introduced by Kierstead and Yang~\cite{KY03} in 2003, though a parameter similar to $\wcol_2(G)$ is already present 
in the work of Chen and Schelp~\cite{CS93}.
This family of parameters gained considerable attention when Zhu~\cite{Zhu09} proved that it captures important and robust notions of sparsity, namely graph classes with bounded expansion and nowhere dense classes. 
Specifically, a class of graphs $\calC$ has bounded expansion if and only if 
there exists a function $f:\NN\to\NN$ such that for every graph $G$ in $\calC$ and every $r\in\NN$ we have $\wcol_r(G) \leq f(r)$. 
Classes of bounded expansion include, in particular, planar graphs, graphs of bounded treewidth, and proper minor-closed classes; 
see the book by Ne\v{s}et\v{r}il and Ossona de Mendez~\cite{sparsity} or the recent lecture notes of Pilipczuk, Pilipczuk, and Siebertz~\cite{notes} for more information on this topic. 

The purpose of this short paper is to give improved bounds on weak coloring numbers for some natural classes of graphs, and to emphasize some open problems about the best possible bounding functions. 
Let us start with a quick review of previous works. 
The $r$-th weak coloring number of graphs with treewidth at most $k$ is at most $\binom{r+k}{k}$ and this is tight, as proved by Grohe, Kreutzer, Rabinovich, Siebertz, and Stavropoulos~\cite{Grohe15}. 
Since treewidth-$2$ graphs are planar, this shows that there are planar graphs $G$ with $\wcol_r(G)$ in $\Omega(r^2)$. 
The best upper bound for planar graphs is $\binom{r+2}{2} (2r+1) \in \Oh(r^3)$, due to Van den Heuvel, Ossona de Mendez, Quiroz, Rabinovich, and Siebertz~\cite{vdHetal17} and proved using a graph decomposition known as ``chordal partition''.\footnote{
  We note that these chordal partitions are precursors of, and an inspiration for, the so-called product structure theorem for planar graphs, which states that every planar graph is a subgraph of the strong product of a treewidth-$8$ graph with a path~\cite{DJMMUW20} (see~\cite{DHJLW21} for a recent survey). A straightforward application of the product structure theorem gives a $\Oh(r^3\log r)$ bound on the $r$-th weak coloring number of planar graphs. We are not aware of any better bound using this approach. 
}  
The latter authors also showed an $\left(g + \binom{r+2}{2}\right) (2r+1)$ upper bound for graphs of Euler genus $g$, which is $\Oh(r^3)$ for fixed $g$.\footnote{Let us recall that the Euler genus of a graph $G$ is the minimum Euler genus of a surface $G$ embeds in, where the sphere $\mathbb{S}_h$ with $h$ handles has Euler genus $2h$, and the sphere $\mathbb{N}_h$ with $h$ crosscaps has Euler genus $h$. We note that the result in~\cite{vdHetal17} is phrased in terms of orientable genus, not Euler genus. However, their proof also works in the nonorientable case.}  
In Table~\ref{tab:state-of-art}, we summarize the state of the art and present our contributions.

 \begin{table}[!ht]
\begin{tabular}{l|cl}     
\toprule
class $\mathcal{C}$ & 
$\max\{\wcol_r(G):G\in\mathcal{C}\}$ \\
  \midrule \\[0.1ex]
outerplanar &$\Theta(r\log r)$& our contribution\\[2ex]
planar & $\Omega(r^2)$ \quad and \quad $\Oh(r^3)$ & 
\begin{tabular}{@{}l}
upper bound: 
Van den Heuvel,\\
Ossona de Mendez, Quiroz,\\
Rabinovich, Siebertz~\cite{vdHetal17}
\end{tabular}\\[4ex]
& $\Omega(r^2\log r)$&our contribution\\[4ex]
Euler genus $g$ & $\Oh_g(r^3)$ & \cite{vdHetal17}\\[4ex]
$K_{3,k}$-minor free & $\Oh_k(r^3)$ & Van den Heuvel and Wood~\cite{vdHW18}\\[4ex]
$K_{s,k}$-minor free ($s\leq k$) & $\Oh_{k}(r^{s+1})$ & \cite{vdHW18}\\[4ex]
treewidth $\leq k$ & $\displaystyle\binom{r+k}{k}$&
\begin{tabular}{@{}l}
Grohe, Kreutzer, Rabinovich, \\
Siebertz, Stavropoulos~\cite{Grohe15}
\end{tabular}\\[4ex]
\begin{tabular}{@{}l}simple\\[0.5ex] treewidth $\leq k$\end{tabular} 
& \begin{tabular}{@{}l}$\Omega_k(r^{k-1}\log r)$\\[0.5ex] and $\Oh(r^{k-1}\log r)$ \end{tabular} 
&our contribution\\[4ex]
$K_k$-minor free &$\Omega_k(r^{k-2})$ \quad and \quad $\Oh(r^{k-1})$ & \cite{Grohe15,vdHetal17}\\[2ex]
\begin{tabular}{@{}l}max.\ degree $\leq\Delta$\\[0.5ex] for $\Delta\geq4$\end{tabular} 
& $\Omega\left(\left(\frac{\Delta-1}{2}\right)^r\right)$&\cite{Grohe15}\\
\bottomrule
\end{tabular}
~\\[1ex]
\caption{Bounds on the maximum $r$-th weak coloring numbers for some graph classes. The subscripts in the asymptotic notations indicate that the hidden constant factors depend on the corresponding parameters.} 
\label{tab:state-of-art}
\end{table}

As a side remark, we note that all the proofs of the upper bounds mentioned above for particular classes of graphs give orderings that are {\em universal}, in the sense that they do not depend on $r$, only on the graphs. 
In fact, Van den Heuvel and Kierstead~\cite{vdHK21} recently proved the existence of universal orderings for weak coloring numbers in classes with bounded expansion:  
There exists a function $f:\NN\to\NN$ such that for every graph $G$ there is an ordering $\sigma^*$ of the vertices of $G$ so that for every $r\geq0$ we have $\wcol_r(G,\sigma^*) \leq f(r,\wcol_{2r}(G))$ (see also~\cite[Chapter 2, Theorem~4.3]{notes}). 

Weak coloring numbers have been used as a tool capturing the right  structures on a way to algorithmic or combinatorial results for various classes of sparse graphs. 
We mention three examples from the literature. 

The \emph{exact distance-$p$ graph} $G^{[\#p]}$ of a graph $G$ is the graph on the same vertex set as $G$ and where $uv$ is an edge in $G^{[\#p]}$ if $u$ and $v$ are at distance exactly $p$ in $G$. 
Van den Heuvel, Kierstead, and Quiroz~\cite{vdHKQ19} proved that for every graph $G$ and every odd natural number $p$, we have
$\chi(G^{[\#p]}) \leq \wcol_{2p-1}(G)$. 
The power of this bound can be observed when we note that for $p=3$ and a planar graph $G$, the best previously known bound was $\chi(G^{[\#3]})\leq 5\cdot 2^{20\, 971\, 522}$.
Using this bound with the bound for planar graphs mentioned above gives  $\chi(G^{[\#3]})\leq \wcol_{5}(G) \leq \binom{5+2}{2}\cdot(2\cdot5+1) = 221$.

\emph{Dimension} is a key measure of a poset's complexity. 
The \emph{cover graph} of a poset is its Hasse diagram taken as an undirected graph, thus two elements $x$ and $y$ of the poset are adjacent in its cover graph if $x$ and $y$ are comparable and there is no third element sandwiched between them in the poset relation.
The \emph{height} of a poset is the maximum size of a chain in the poset.
Joret, Ossona de Mendez, Micek, and Wiechert~\cite{JMOdMW19}  proved that 
$\dim(P) \leq 4^{c}$, where $c=\wcol_{3h-3}(G)$, $h$ is the height of $P$, and $G$ is the cover graph of $P$.
This result implies and generalizes a series of previous works on poset dimension, and with a much simpler proof, suggesting that weak coloring numbers are the right tool to use in this context.  
Weak coloring numbers are also used in~\cite{JMOdMW19} to show that the property of being nowhere dense for a graph class can be captured by looking at the dimension of posets whose cover graphs are in the class.

Our third example concerns a generalization of domination and independence numbers. 
The {\em $k$-domination number $\dom_k(G)$} of a graph $G$ is the minimum size of a vertex subset $X$ of $G$ such that every vertex of $G$ is at distance at most $k$ from $X$. 
The {\em $d$-independence number $\alpha_d(G)$} of $G$ is the maximum size of a subset $S$ of vertices of $G$ such that every two vertices in $S$ are at distance strictly greater than $d$ in $G$. 
Observe that the $2k$-independence number is a lower bound on the $k$-domination number, that is, $\alpha_{2k}(G) \leq \dom_k(G)$. 
Dvo\v{r}\'ak~\cite{D13} showed that the two parameters are in fact tied to each other when the weak coloring numbers are bounded: $\dom_k(G)\leq (\wcol_{2k}(G))^2 \cdot \alpha_{2k}(G)$. 
In a subsequent work, Dvo\v{r}\'ak~\cite{D19} also proved that $\dom_k(G)\leq 4(\wcol_{k}(G))^4 \wcol_{2k}(G) \cdot \alpha_{2k}(G)$, which is a better upper bound when $\wcol_{k}(G)$ is much smaller than $\wcol_{2k}(G)$. 
The proofs of these results are algorithmic and lead to approximation algorithms for computing $\dom_k(G)$ and $\alpha_{2k}(G)$ in graph classes with bounded expansion. 
Finally, we note that weak coloring numbers were also used in~\cite{DDFKLPPRVS16} to obtain small kernels for the problem of computing the $k$-domination number for graph classes with bounded expansion, see also~\cite{EGKKPRS17} for related results.

Let us now turn to our contributions. 
First we consider paths, for which we determine the optimal bound up to an additive constant of $1$. 
All logarithms in this paper are in base $2$, unless otherwise stated.  

\begin{theorem}
\label{thm:path}
\hfill
\begin{enumerate}
\item For every integer $r\geq1$ and every path $P$, 
\[
\wcol_r(P) \leq \left\lceil\log r\right\rceil + 2.
\]
\item For every integer $r\geq1$ and every path $P$ on at least $2r$ vertices, 
\[
\wcol_r(P) \geq \lceil\log (r+1)\rceil+1.
\]
\end{enumerate}
\end{theorem}
We remark that the upper and lower bounds in Theorem~\ref{thm:path} coincide when $r$ is a power of $2$. 

Our main result concerns graphs of simple treewidth at most $k$. 
Simple treewidth is a variant of treewidth introduced by Knauer and Ueckerdt~\cite{KU} (see also~\cite{LW-thesis}).  
One way of defining the treewidth of a graph $G$ is as follows: This is the smallest nonnegative integer $k$ such that $G$ is a subgraph of a $k$-tree,  where a {\em $k$-tree} is any graph that can be obtained by starting with a $(k+1)$-clique and repeatedly choosing an existing $k$-clique and adding a new vertex adjacent to all vertices of the clique. 
If we add the extra requirement that a $k$-clique is never chosen more than once, then the resulting graphs are {\em simple $k$-trees}; accordingly, the smallest nonnegative integer $k$ such that $G$ is a subgraph of a simple $k$-tree is the {\em simple treewidth $\stw(G)$} of $G$. 
It is easy to see that 
\[\tw(G)\leq \stw(G)\leq \tw(G)+1.\]
For $k=1,2,3$, graphs of simple treewidth at most $k$ coincide with the following graph classes: 
\begin{itemize}
    \item[]\hspace{-0.8cm}$k=1$ \quad disjoint unions of paths;
    \item[]\hspace{-0.8cm}$k=2$ \quad outerplanar graphs;
    \item[]\hspace{-0.8cm}$k=3$ \quad planar graphs of treewidth at most $3$, i.e.\ subgraphs of stacked triangulations.
\end{itemize}
We note that simple $3$-trees are known under the names of stacked triangulations, planar $3$-trees, and Apollonian networks in the literature. 

The following theorem is our main result. 

\begin{theorem}
\label{thm:stw} 
For all integers $k\geq 1$ and $r\geq 1$ we have 
\[
  \wcol_r(G) \leq \Oh(r^{k-1} \log r) 
\]
for every graph $G$ of simple treewidth at most $k$. 
Furthermore, there is such a graph $G$ satisfying 
\[
  \wcol_r(G) \geq \frac{r^{k-1} \ln r}{k!} = \Omega_k(r^{k-1} \log r). 
\]
\end{theorem}

\begin{corollary}\hfill
\begin{enumerate}
\item Outerplanar graphs have $r$-th weak coloring numbers in $\Oh(r\log r)$ and this bound is tight.
\item There is a family of planar graphs with $r$-th weak coloring numbers in $\Omega(r^{2}\log r)$.
\end{enumerate}
\end{corollary}

An open problem that we find particularly intriguing is to determine the right asymptotics for the maximum $r$-th weak coloring numbers of planar graphs. 
The current best bounds are $\Omega(r^{2}\log r)$ and $\Oh(r^3)$. 
We believe that the lower bound is the right order of magnitude:

\begin{conjecture}
\label{conj:planar}
Planar graphs have $r$-th weak coloring numbers in $\Oh(r^{2}\log r)$.  
\end{conjecture}

The paper is organized as follows. 
In Section~\ref{sec:paths} we prove \Cref{thm:path} about paths. 
Then in Section~\ref{sec:stw} we prove \Cref{thm:stw} about graphs with bounded simple treewidth. 
Finally in Section~\ref{sec:open_problems} we discuss a number of open problems about weak coloring numbers. 

\section{Paths}
\label{sec:paths}

A \emph{rooted forest} is a disjoint union of rooted trees.
The \emph{height} of a rooted forest $F$ is the maximum number of vertices on a path from a root to a leaf in $F$. 
For two vertices $u$, $v$ in a rooted forest $F$, we say that $u$ is an \emph{ancestor} of $v$ in $F$
if $u$ lies on a path from a root to $v$ in $F$.
The \emph{closure} of $F$, denoted by $\clos(F)$, is the graph with vertex set $V(F)$ and edge set $\set{\set{v,w}\mid v \text{ is an ancestor of $w$ in $F$}}$.
The \emph{treedepth} of a graph $G$, denoted by $\td(G)$, is the minimum height of a rooted forest $F$ such that
$G\subseteq \clos(F)$.
The following is a folklore observation, see e.g.~\cite[Lemma~6.5]{sparsity}.
\begin{lemma}
\label{lem:td}
For every graph $G$ we have the equality
\[
\wcol_{\infty}(G) = \td(G).
\]
\end{lemma}

The treedepth of the path $P_n$ on $n$ vertices is $\left\lceil\log(n+1)\right\rceil$. 
See~Figure~\ref{fig:path-td}.
This formula follows from a simple recursion: $\td(P_0)=0$, $\td(P_1)=1$ and 
$\td(P_n) = 1+ \min_{1\leq i \leq n}\max\set{\td(P_{i-1}),\td(P_{n-i})} = 1 + \td\left(P_{\lceil\frac{n-1}{2}\rceil}\right)$.
\begin{figure}[!h]
    \centering
    \includegraphics{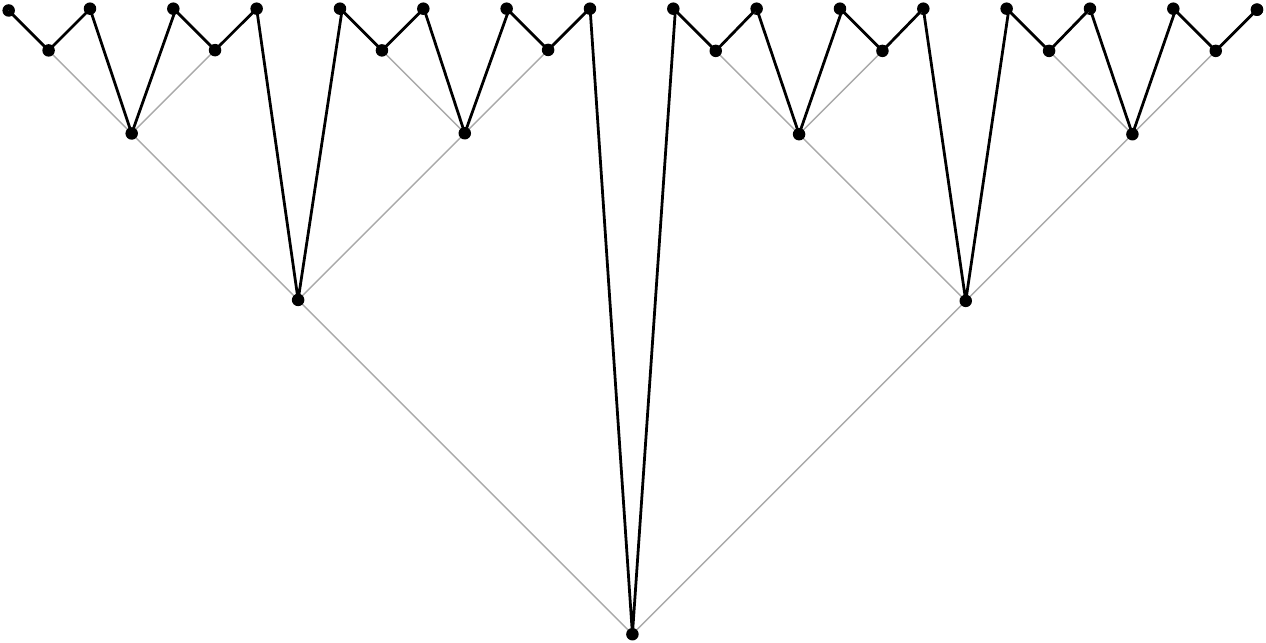}
    \caption{The treedepth of a path is logarithmic in the path length.}
    \label{fig:path-td}
\end{figure}

\begin{proof}[Proof of Theorem~\ref{thm:path}]
First we prove the upper bound.
Let $r\geq1$ and let $P$ be a path on $n$ vertices. 
Enumerate the vertices of $P$ from one end to the other:
$v_1,\ldots,v_n$.
If $r=1$, then we just define an ordering $\sigma$ to follow the enumeration and
clearly $|\WReach_{1}[P,\sigma,v]|\leq 2$, for every vertex $v$ of $P$.

Thus, suppose that $r\geq2$.
Let $V_0=\set{v_i\in V(P)\mid i = 0 \pmod r}$. 
Note that the components of $P-V_0$ are paths on at most $r-1$ vertices.
Let $\sigma$ be an ordering of $V(P)$ so that 
\begin{enumerate}
\item for every $v\in V_0$ and $w\in V(P)- V_0$, we have $v\leq_{\sigma} w$;
\item
\label{item:wcol-infty} 
for each component $P'$ of $P-V_0$, the restriction $\restr{\sigma}{V(P')}$ is witnessing that $\wcol_{\infty}(P') \leq \left\lceil\log r\right\rceil$.
\end{enumerate}

(We can do this by Lemma~\ref{lem:td}.)  
We claim that for every vertex $v$ in $P$ we have $|\WReach_{r}[P,\sigma,v]| \leq \left\lceil\log r\right\rceil + 2$.

Consider first $v\in V_0$. 
By the construction of $\sigma$ the only vertices $v$ can weakly $\infty$-reach are in $V_0$. 
But there are at most two vertices from $V_0$ distinct from $v$ in distance at most $r$ from $v$ in $P$. 
Therefore, $|\WReach_{r}[P,\sigma,v]|\leq 3 \leq \left\lceil\log r\right\rceil + 2$.

Consider now $v\in V(P)- V_0$ and let $P'$ be the component of $P-V_0$ containing $v$.
Again by the construction of $\sigma$ the only vertices $v$ can weakly $\infty$-reach are in $P'$ and $V_0$.
By~\ref{item:wcol-infty}, $|\WReach_{\infty}[P,\sigma,v]\cap V(P')|\leq \left\lceil\log r\right\rceil$. 
And again there are at most two vertices from $V_0$ in distance at most $r$ from $v$ in $P$. 
Thus indeed $|\WReach_{r}[P,\sigma,v]|\leq \left\lceil\log r\right\rceil + 2$, as desired.

Now we switch to the proof of the lower bound. 
Consider a path $P$ on at least $2r$ vertices and a linear order $\sigma$ on $V(P)$. 
We are going to show that $\wcol_{r}(P,\sigma) \geq \lceil\log (r+1)\rceil+1$.
Let $v_0$ be the $\sigma$-minimum vertex of $P$. 
Since $P$ has at least $2r$ vertices, one of the components of $P-v_0$ 
contains at least $r$ vertices.
Fix a subpath $Q$ of $P$ in such a component on exactly $r$ vertices with one endpoint of $Q$ being a neighbor of $v_0$ in $P$.
Since $\wcol_{\infty}(Q) = \lceil\log (r+1)\rceil$, there must be a vertex $q$ in $Q$ such that 
$|\WReach_{\infty}[Q,\sigma,q]|\geq \lceil\log (r+1)\rceil$. 
Note also that $\WReach_{\infty}[Q,\sigma,q] = \WReach_{r-1}[Q,\sigma,q]$ since $Q$ is a path on $r$ vertices. 
Furthermore, $v_0$ is at distance at most $r$ to $q$ and $v_0$ is the $\sigma$-minimum vertex so 
$v_0$ is $r$-weakly reachable from $q$ in $P$. Therefore,
\[
|\WReach_{r}[P,\sigma,q]| \geq |\WReach_{\infty}[Q,\sigma,q]| + 1 \geq \lceil\log (r+1)\rceil +1.
\]
\end{proof}

\section{Simple treewidth}
\label{sec:stw}

We begin this section by giving definitions of treewidth and simple treewidth that are slightly different but equivalent to the ones given in the introduction. 
These definitions will be more convenient for our purposes.  
Let $G$ be a graph. 
A \emph{tree-decomposition} of $G$ is a pair $(T,\calB)$ where $T$ is a tree and $\calB=(B_t)_{t\in V(T)}$ is a family of subsets of $V(G)$, satisfying
\begin{enumerate}
\item \label{prop-1-td} for each $v\in V(G)$ the set $\set{t\in V(T)\mid v\in B_t}$ induces a non-empty subtree of $T$;
\item for each $uv\in E(G)$ there exists $t\in V(T)$ with $u,v\in B_t$.
\end{enumerate}
We usually call the sets $B_t$, the \emph{bags} of $(T,\calB)$.
The \emph{width} of a tree-decomposition $(T,\calB)$ is $\max_{t\in V(T)} |B_t|-1$.
The \emph{treewidth} of $G$, denoted by $\tw(G)$, is the least integer $k$ such that $G$ has a tree-decomposition of width at most $k$.

For $k\geq1$, a tree-decomposition $(T,\calB)$ of $G$ is $k$-\emph{simple} if $(T,\calB)$ is of width at most $k$ and for every $X\subset V(G)$ with $|X|=k$, we have that $X\subset B_t$ for at most two distinct $t\in V(T)$. 
The \emph{simple treewidth} of $G$, denoted by $\stw(G)$, 
is the least integer $k$ such that $G$ has a $k$-simple tree-decomposition. 

The following lemma is proved in~\cite{LW-thesis}, we give a proof to keep the paper self contained. 

\begin{lemma}
\label{lem:minor_stw}
Let $G$ be a graph and $H$ be a minor of $G$. Then $\stw(H)\leq \stw(G)$.
\end{lemma}
\begin{proof}
Let $(T,\calB)$ be a $k$-simple tree-decomposition of $G$. 
It is enough to prove the lemma in the case where $H$ is obtained from $G$ by using one of the following three operations once: vertex deletion, edge deletion, or edge contraction. 
For the edge-deletion, it is clear that $(T,\calB)$ is still a $k$-simple tree-decomposition of $H$.
For the vertex deletion, when we remove $v$ from $G$, it is enough to remove $v$  from all the bags in $\calB$. 
The resulting $(T,\calB)$ is a $k$-simple tree-decomposition of $H$.

For the third operation, say $H$ is obtained from $G$ by contracting the edge $uv$ into a vertex $w$. 
Let $(T,\calB')$ be obtained from $(T,\calB)$ by replacing each occurrence of $u$ and $v$ with $w$ in all the bags. 
Note that $(T,\calB')$ is a tree-decomposition of $H$ of width at most $k$, though it is not necessarily $k$-simple. 
Next, apply the following reduction operation on $(T,\calB')$ as long as it is possible: 
If there is an edge $st$ of $T$ such that $B'_s \subseteq B'_t$, contract the edge $st$ of $T$ into $t$ (whose bag is still $B'_t$). 
Observe that this preserves the fact that $(T,\calB')$ is a tree-decomposition of $H$ of width at most $k$. 
Let $(T^*,\calB^*)$ denote the resulting tree-decomposition of $H$ when the process stops. 
We claim that $(T^*,\calB^*)$ is $k$-simple. 
Indeed, suppose $X$ is a subset of $k$ vertices of $H$. 
If $w \notin X$, then clearly $X$ appears in at most as many bags of $(T^*,\calB^*)$ as of $(T,\calB)$. 
If $w \in X$, then consider $(T,\calB')$ (before any reduction operation is applied). 
The two sets $X_u:=(X-\{w\})\cup\{u\}$ and $X_v:=(X-\{w\})\cup\{v\}$ each appears in at most two bags of $(T,\calB)$. 
If one of $X_u, X_v$ appears in no such bags, then clearly $X$ appears in at most two bags of $(T,\calB')$, and thus of $(T^*,\calB^*)$, as well. 
Thus assume $X_u\subseteq B_{t_u}$ and $X_v\subseteq B_{t_v}$ for some nodes $t_u$, $t_v$ in $T$.
Since $(T,\calB)$ is a tree-decomposition we know that $X$ is in all the bags $B_t$ for $t$ in the path $t_uTt_v$.
Since $uv\in E(G)$ we know that there exists $t$ in the path $t_uTt_v$ such that $\set{u,v}\subseteq B_t$. 
We fix such $t$ and we have $B_t\supseteq(X-\set{w})\cup\set{u,v}$ and 
since $|(X-\set{w})\cup\set{u,v}|=k+1$ we must have $B_t=(X-\set{w})\cup\set{u,v}$. 
In particular, $B'_{t}=X$. 
Since $X_u$ appears in at most two bags and $X_v$ appears in at most two bags of $(T,\calB)$ 
but each of them appears in $B_t$, we conclude that $X$ appears in at most three bags of $(T,\calB')$.
We claim that if $X$ appears in any other bag of $(T,\calB')$ than $B'_t$, then the node $t$ will be contracted in the process.
Indeed, if $X$ appears in any bag of $(T,\calB')$ other than $B'_t$, 
then by property~\ref{prop-1-td} of a tree-decomposition it must appear in a neighboring bag $B'_{t'}$ where $tt'\in E(T)$.
In this case we have $B'_t = X \subseteq B'_{t'}$ so this edge of $T$ would be contracted in the reduction.
This proves that $(T^*,\calB^*)$ has at most two bags containing $X$, as desired.
Thus, $(T^*,\calB^*)$ is $k$-simple.
\end{proof}

The \emph{distance} $\dist_G(u,v)$ between two vertices $u$ and $v$ in a graph $G$ is the length of a shortest path connecting $u$ and $v$ (if there is no path between $u$ and $v$ in $G$ then $\dist_G(u,v)=\infty$). 
If $r$ is a vertex in a connected graph $G$ and $L_i=\set{v\in V(G)\mid \dist_G(r,v)=i}$ for all integers $i\geq0$, 
then $(L_0,L_1,\ldots)$ is called a \emph{BFS-layering} of $G$. 

The following lemma follows quickly from Lemma~\ref{lem:minor_stw}. 
It is also used in~\cite[Lemma 14]{BDJM20}, where it is stated without proof. 
We include a proof for completeness.

\begin{lemma}
\label{lem:stw_layering}
Let $G$ be a connected graph with $\stw(G)=k \geq 1$ and let $(L_0,L_1,\ldots)$ be a BFS-layering of $G$. 
Then $\stw(G[L_i]) \leq k-1$ for every $i\geq0$.
\end{lemma}
\begin{proof}
Fix an integer $i\geq0$. 
If $i=0$, then $G[L_i]$ is a one-vertex graph and $\stw(G[L_i])=0\leq k-1$. 
If $L_i = \emptyset$, then the claim trivially holds. 
Thus, assume that $i>0$ and $L_i \neq \emptyset$. 
Consider the set $L=L_0\cup\cdots\cup L_{i-1}$. 
Clearly, $L$ is non-empty and $G[L]$ is connected. 
Let $H$ be a graph obtained from $G$ by contracting all the vertices in $L$ to a single vertex and removing all the vertices in $\bigcup_{j>i} L_j$.
Since $H$ is a minor of $G$ we have $\stw(H) \leq \stw(G) \leq k$.
The single vertex of $H$ that results from the contraction of $L$ we denote by $r$.

Let $(T,\calB)$ be a $k$-simple tree decomposition of $H$ 
with $T$ being minimal possible under taking subgraphs. 
We claim that $r\in B_t$ for every $t\in V(T)$ and $B_t\in\calB$.
Consider a leaf (vertex of degree one) $t$ of $T$ and let $t'$ be the only neighbor of $t$ in $T$.
If $r \not\in B_t$, then $B_t \subseteq B_{t'}$ as all the vertices $v\in V(H)- \set{r}$ are adjacent to $r$, so no vertex $v$ can lie only in $B_t$ and nowhere else. 
But when $B_t \subseteq B_{t'}$ then we can remove $t$ from $T$ and we would still have a tree-decomposition of $H$ contradicting the choice of $(T,\calB)$.
Therefore, $r$ lies in all the leaf bags of $(T,\calB)$ so it must be in all the bags.

Let $B'_t=B_t-\set{r}$ for each $t$ in $T$ and $\calB'=\set{B'_t}_{t\in V(T)}$.
We obtain a tree-decomposition $(T,\calB')$ of $G[L_i]$ and we claim that it is $(k-1)$-simple.
For each $t$ in $T$ we have $|B'_t| = |B_t| -1 \leq k-1$, so $(T,\calB')$ is of width at most $k-1$.
Consider any set $X \subset L_i$ with $|X|=k-1$. 
We claim that $X\subseteq B'_t$ for at most two nodes $t$ in $T$.
Indeed, for each such node $t$ we have 
$X\cup\set{r} \subseteq B_t$ but since $(T,\calB)$ is $k$-simple there are at most two such nodes.
This proves that $(T,\calB')$ is $(k-1)$-simple.
\end{proof}

We also need the following easy lemma about ``shadow completeness'' in BFS layerings of chordal graphs. 
(A graph is {\em chordal} if it has no induced cycle of length at least $4$.) 

\begin{lemma}[see e.g.~\cite{KP08}]
\label{lem:shadow_complete}
Let $G$ be a connected chordal graph and let $(L_0,L_1,\ldots)$ be a BFS-layering of $G$. 
Suppose $H$ is a connected component of $G[\bigcup_{j \geq i} L_j]$ for some $i \geq 1$. 
Then the set of neighbors of $V(H)$ in $L_{i-1}$, which we call the {\em shadow} of $H$, is a clique in $G$. 
\end{lemma}

We may now turn to the proof of Theorem~\ref{thm:stw}.

\begin{proof}[Proof of Theorem~\ref{thm:stw}]
For convenience, let $f(r,k):=(r+1)^{k-1}(\left\lceil\log r\right\rceil + 2)$, 
for all integers $r\geq1$, $k\geq1$.
We are going to prove by induction on $k$ that 
$\wcol_r(G) \leq f(r,k)$.
For the base case, recall that graphs of simple treewidth at most $1$ are disjoint unions of paths, thus by Theorem~\ref{thm:path} we have $\wcol_r(G) \leq \left\lceil\log r\right\rceil + 2 = f(r,1)$ for all such graphs $G$ and every integer $r\geq1$.

For the induction step, let $k\geq2$, 
let $G$ be a graph with simple treewidth at most $k$. 
Consider a $k$-simple tree decomposition of $G$.
Note that adding edges to $G$ does not decrease its weak coloring numbers. 
Therefore, we may assume that each bag of the tree decomposition induces a clique in $G$. 
Thus, $G$ is chordal and connected.

Let $(L_0,L_1,\ldots)$ be a BFS-layering of $G$. 
Let $q\geq 0$ be maximum such that $L_q \neq \emptyset$. 
By Lemma~\ref{lem:stw_layering} we have $\stw(G[L_i])\leq k-1$ so by induction, 
we fix an ordering  $\sigma_i$ of $L_i$ witnessing that $\wcol_r(G[L_i]) \leq f(r,k-1)$, 
for each $i\in \{0,1,\dots, q\}$. 
Let $\sigma$ be the ordering of $V(G)$ obtained by concatenating these orderings, that is, 
$\sigma=\sigma_0 \cdots \sigma_q$.

We will show that
\[
|\WReach_r[G,\sigma,v]| \leq (r+1)\cdot f(r,k-1)
\]
for each vertex $v$ in $G$. 
This will complete the induction step.
Let $v\in V(G)$ and suppose $v\in L_i$. 
Since vertices $r$-weakly reachable from $v$ are in distance at most $r$ from $v$ in $G$ 
and since vertices in layers $L_{i+1}, L_{i+2},\ldots$ are greater than $v$ in $\sigma$, 
we have that $\WReach_r[G,\sigma,v| \subseteq \bigcup_{j} L_j$ where 
$j\in\set{i-r,\ldots,i}$ and $j\geq0$.
It is thus enough to show that, for each such $j$ we have
at most $f(r,k-1)$ vertices from $L_j$ that are $r$-weakly reachable from $v$ in $G$. 
Fix such an index $j$. 

Let $w\in L_j$ be $r$-weakly reachable from $v$ in $G$.
Consider a $vw$-path witnessing that $w$ is $r$-weakly reachable from $v$ in $G$, and let $P$ be a shortest such path. 
Observe that $P$ does not enter layer $L_{j-1}$ (if $j>0$). 
Also, once $P$ enters layer $L_j$, it stays in $L_j$. 
Indeed, if not, then there is an $xy$-subpath of $P$ of length at least $2$, with $x, y\in L_j$, and with all internal vertices in the same component of $L_{j+1} \cup \cdots \cup L_q$. 
By Lemma~\ref{lem:shadow_complete}, $xy$ is an edge in $G$, 
and hence $P$ could be shortcutted using the edge $xy$, a contradiction. 

If $j=i$, it follows that $P$ is fully contained in $L_i$, 
so $w\in \WReach_r[G[L_i],\sigma_i,v]$.
Thus by induction $|\WReach_r[G,\sigma,v]\cap L_i| \leq |\WReach_r[G[L_i],\sigma,v]| \leq f(r,k-1)$ as desired. 
If $j<i$, let $C_j \subseteq L_j$ be the shadow of the connected component of $G[L_{j+1}\cup \cdots \cup L_q]$ containing $v$. 
Since $G$ is chordal, $C_j$ induces a clique in $G$ by Lemma~\ref{lem:shadow_complete}. 
Then $P$ enters $L_j$ in some vertex $v'_j \in C_j$, and the $v'_jw$-subpath $Q$ of $P$ has length at most $r-(i-j) \leq r-1$. 
Let $v_j$ be the greatest vertex of $C_j$ in $\sigma_j$. 
Adding $v_j$ at the beginning of $Q$ in case $v'_j\neq v_j$, 
we see that $w$ is $(r-(i-j)+1)$-weakly reachable from $v_j$ in $G[L_j]$. 
(Here we use that $v_j$ is the rightmost vertex of $C_j$ in $\sigma_j$.)
Thus by induction again $|\WReach_r[G,\sigma,v]\cap L_j| \leq f(r,k-1)$ as desired. 



Now we turn to the lower bound. 
Let $g(0,k) := 1$ for all $k\geq1$, let 
$g(r,1) := \lceil\log r\rceil+1$ for all $r\geq1$, and let 
\[
g(r,k):= \sum_{i=0}^r g(i, k-1),  
\]
for all $r\geq1$ and $k\geq 2$.

For each $r \geq 1$ and $k \geq 1$ we construct a graph $G_{r,k}$ with $\stw(G_{r,k})=k$ and $\wcol_r(G_{r,k}) \geq g(r,k)$.  

The graph $G_{r,k}$ is defined inductively on $k$. 
For $k=1$, the graph is a path on $2r$ vertices. 
For $k \geq 2$, the graph $G_{r,k}$ is obtained as follows. 
First, create the root vertex $s$.  
We are going to define the graph layer by layer, 
where layer $i$ corresponds to all vertices at distance exactly $i$ from $s$ in $G_{r,k}$. 
Let $i\in\set{0,1,\ldots,r-1}$ and assume that layer $i$ is already created.
For each vertex $v$ in layer $i$, 
create $g(r,k)$ disjoint copies of $G_{i+1,k-1}$, 
which we call the {\em private copies} of $v$, 
and make all their vertices adjacent to $v$. 
This defines layer $i+1$. 

Let us show that $\stw(G_{r,k}) \leq k$, by induction on $k$. 
For the base case $k=1$, this is clear since $G_{r,1}$ is a path. 
For the inductive case, $k\geq 2$, we construct a $k$-simple tree decomposition of $G_{r,k}$ step by step following the inductive definition of $G_{r,k}$. 
First, create the root bag containing only the root $s$.  
Next, for $i=0, 1, \dots, r-1$ and for each vertex $v$ in layer $i$, 
consider a node $t$ of the tree indexing the tree decomposition we are constructing whose bag contains $v$. 
For each private copy $H$ of $G_{i+1,k-1}$ belonging to $v$, 
consider a $(k-1)$-simple tree decomposition $(T', \mathcal{B'})$ of $H$, 
which exists by induction, 
add $v$ to every bag, 
and link that tree decomposition to the main one by adding an edge connecting one of the nodes of $T'$ to $t$. 
We claim that this operation keeps the main tree decomposition $k$-simple. 
Clearly every new bag has size at most $k+1$. 
For every subset $X$ of $k$ vertices, 
let us show there are at most two bags containing $X$. 
If $X$ is not fully contained in $V(H) \cup \{v\}$, this follows from the fact that the main tree decomposition was $k$-simple before adding the new bags. 
If $X\subseteq V(H) \cup \{v\}$ and $v\not\in X$, 
this follows from the fact that $X$ appears in the same bags as is $(T',\calB')$ and $|X|=k\geq k-1$ so it appears in at most two bags.
If $X\subseteq V(H) \cup \{v\}$ and $v\in X$, then
there are most two bags containing $X-\{v\}$ in the tree decomposition $(T', \mathcal{B'})$ and so it is in our resulting tree-decomposition.

It remains to prove that $\wcol_r(G_{r,k}) \geq g(r,k)$, which we show by induction on $k$ again. 
For the base case $k=1$, this follows from Theorem~\ref{thm:path}.\footnote{We remark that we could even have set $g(r,1) := \lceil\log (r+1)\rceil+1$ thanks to Theorem~\ref{thm:path}, however this does not simplify the proof of the inductive case.}  
For the inductive case, suppose $k\geq 2$. 
Let $\sigma$ be any ordering of $V(G_{r,k})$. 
We define $r+1$ vertices $v_0, v_1, \dots, v_r$ as follows. 
First we set $v_0$ to be the root of $G_{r,k}$.
Then, for $i\in\set{0, \dots,r-1}$, we consider the $g(r,k)$ private copies of $v_{i}$. 
If each copy has at least one vertex before $v_{i}$ in $\sigma$ then we are done, since then there are $g(r,k)$ vertices that are $1$-weakly reachable from $v_{i}$. 
Thus we may assume that some private copy of $v_{i}$ has all its vertices after $v_{i}$ in $\sigma$. 
Consider such a copy, call it $H_{i+1}$, and let $\sigma_{i+1}$ denote the ordering $\sigma$ restricted to $H_{i+1}$. 
Using induction, let $v_{i+1}$ be a vertex of $H_{i+1}$ such that at least $g(i+1,k-1)$ vertices of $H_{i+1}$ are $(i+1)$-weakly reachable from $v_{i+1}$ in $H_{i+1}$ under $\sigma_{i+1}$.  

It follows that for each $i\in\set{0, 1,\dots,r}$, 
the vertex $v_r$ can weakly reach $v_i$ going through $(r-i)$ edges 
and there are $g(i,k-1)$ vertices $i$-reachable from $v_i$.
Thus in total, there are at least $\sum_{i=0}^r g(i, k-1)=g(r,k)$ vertices that are $r$-weakly reachable from $v_r$, as desired. 

It only remains to show that 
\[
  g(r,k) \geq \frac{r^{k-1} \ln r}{k!},
\]
for all $r, k\geq 1$, which we do by induction on $r+k$. 
This is clearly true if $r=1$ or $k=1$. 
Also, if $r=2$, then the claim holds because $g(2, k)\geq 1 \geq \frac{2^{k-1}\ln 2}{k!}$ for $k \geq 2$. 
Next, let us consider the case $r \geq 3$ and $k=2$. 
We have 
\[
g(r, 2) = \sum_{x=0}^r g(x, 1) 
\geq (r+1) + \sum_{x=2}^r \ln x
\geq (r+1) + \int_1^r \ln x \;\mathrm{d}x
\geq \frac{r \ln r}{2}. 
\]
Now we consider the case $r\in \{3, 4\}$ and $k\geq 3$.  
Observe that $g(r,k) \geq g(3,3)\geq 4$. 
On the other hand, in our range of possible values for $r$ and $k$, the function $\frac{r^{k-1} \ln r}{k!}$ is maximized for $r=4$ and $k=3$, in which case its value is $16\ln(4)/6 \simeq 3.697$. 
Hence, it follows that $g(r,k) \geq \frac{r^{k-1} \ln r}{k!}$, as desired.

It remains to consider the case $r\geq 5$ and $k\geq 3$.  
Then, 
\[
  g(r,k) \geq  \sum_{x=2}^r g(x, k-1) 
  \geq    
  \frac{1}{(k-1)!} \sum_{x=2}^r x^{k-2} \ln x  
  \geq \frac{1}{(k-1)!} \int_{1}^r x^{k-2} \ln x \;\mathrm{d}x.  
\]
We have
\[
  \int_{1}^r x^{k-2} \ln x \;\mathrm{d}x 
  = \frac{x^{k-1}\ln x}{k-1} - \frac{x^{k-1}}{(k-1)^2} \Big|_1^r 
  \geq \frac{r^{k-1}\ln r}{k-1} - \frac{r^{k-1}}{(k-1)^2}. 
\]
Note that $(k-1)\ln r \geq (k-1) \frac{3}{2} \geq k$ since $r\geq 5$ and $k\geq 3$. 
Thus, 
\begin{align*}
\frac{r^{k-1}\ln r}{k-1} - \frac{r^{k-1}}{(k-1)^2} 
&= r^{k-1}\ln r \left(\frac{1}{k-1} - \frac{1}{(k-1)^2\ln r}  \right) \\
&\geq r^{k-1}\ln r \left(\frac{1}{k-1} - \frac{1}{(k-1)k}  \right) \\
&= \frac{r^{k-1} \ln r}{k}. 
\end{align*}
Hence, 
\[
  g(r,k) \geq \frac{1}{(k-1)!} \frac{r^{k-1} \ln r}{k} 
  = \frac{r^{k-1} \ln r}{k!},
\]
as desired. 
\end{proof}

\section{Open problems}
\label{sec:open_problems}

Let us first repeat the open problem about planar graphs discussed in the introduction: 

\begin{problem}
What is the asymptotics of the maximum of $\wcol_r(G)$ when $G$ is planar?
It is known to be $\Omega(r^2\log r)$ and $\Oh(r^3)$. 
We conjecture $\Theta(r^2\log r)$. 
\end{problem}

There are also gaps remaining for classes generalizing planar graphs. 
We mention two problems. 

\begin{problem}
For fixed $k\geq 1$, what is the asymptotics of the maximum of $\wcol_r(G)$ when $G$ has no $K_k$ minor?
It is known to be $\Omega_k(r^{k-2})$ and $\Oh(r^{k-1})$.
\end{problem}

\begin{problem}
For fixed $k\geq s \geq  4$, what is the asymptotics of the maximum of $\wcol_r(G)$ when $G$ has no $K_{s,k}$ minor?
It is known to be $\Omega_k(r^{s})$ and $\Oh_k(r^{s+1})$. Van den Heuvel and Wood~\cite{vdHW18} conjecture $\Theta_k(r^{s})$. 
\end{problem}

We note that graphs $G$ of Euler genus $g$ have no $K_{3,k}$ minor for some $k=\Oh(g)$ and satisfy $\wcol_r(G) \in \Oh_g(r^3)$~\cite{vdHetal17}. 
Van den Heuvel and Wood~\cite{vdHW18} subsequently proved an upper bound of $\Oh_k(r^3)$ for all graphs with no $K_{3,k}$ minor, which motivates their  conjecture above. 

While this paper is focused on weak coloring numbers, 
there are a few other closely related parameters that are good to have in mind when studying weak coloring numbers. 
We conclude this paper with a brief discussion of these parameters and how they relate to weak coloring numbers.  

The {\em $r$-th strong coloring number} $\scol_r(G)$ of a graph $G$ is defined exactly as its weak counterpart, except that now vertex $u$ is \emph{strongly $r$-reachable} from vertex $v$ in an ordering $\sigma$ if $u <_{\sigma} v$ and
there exists an $u$--$v$ path of length at most $r$ such that for every {\em internal} vertex $w$ on the path, $v <_{\sigma} w$. 
Then, 
\[
\scol_r(G) \leq \wcol_r(G) \leq (\scol_r(G))^r, 
\]
where the left inequality is obvious, and the right one was shown by Kierstead and Yang~\cite{KY03}. 
In particular, a class of graphs has bounded weak coloring numbers if and only if it has bounded strong coloring numbers, and thus having bounded strong coloring numbers also captures the notion of bounded expansion. 
However, for specific graph classes, the bounds on strong coloring numbers can be significantly smaller than for weak coloring numbers. 
For instance, $\scol_r(G)\leq 5r+1$ for every planar graph $G$~\cite{vdHetal17}. 
See~\cite{vdHetal17} for more on strong coloring numbers. 

We would like to advertise one intriguing question about strong coloring numbers. 
An important refinement of the notion of bounded expansion for a graph class is that of {\em polynomial expansion}. 
By a result of Dvo{\v{r}}{\'a}k and Norin~\cite{DN16}, for monotone\footnote{meaning closed under subgraphs} classes this is equivalent to having strongly sublinear-size separators, or equivalently, strongly sublinear treewidth: There exists $\epsilon >0$ such that $\tw(G) \in \Oh(n^{1-\epsilon})$ for all $n$-vertex graphs $G$ in the class. 
This includes planar graphs, graphs excluding a $K_k$ minor, intersection graphs of touching balls in $\mathbb{R}^d$, and several other classes of graphs; see in particular~\cite{DMN20} for examples of such classes that are defined geometrically. 
While it is known (and easy to show, see e.g.~\cite{ER18}) that graph classes with strong coloring numbers polynomially bounded in $r$ have polynomial expansion, the converse is open: 

\begin{problem}
Is it true that for every monotone graph class $\mathcal{C}$ with polynomial expansion, there exists $d > 0$ such that, for all integers $r\geq 1$ and all graphs $G \in \mathcal{C}$, we have $\scol_r(G) \in\Oh(r^d)$?
\end{problem}

This problem is discussed in~\cite{ER18, DMN20}.  
We remark that if we replace strong coloring numbers with weak coloring numbers in the above open problem, then the answer is negative: In~\cite{Grohe15}, a class with polynomial expansion and super-polynomial weak coloring numbers is presented. 

A second parameter related to weak coloring numbers is the {\em $p$-centered chromatic number $\chi_p(G)$} of a graph $G$, where $p\geq 1$ is an integer. This is the smallest number of colors in a vertex coloring of $G$ such that, for every connected subgraph $H$ of $G$, either there is a color which appears exactly once among the colors on vertices of $H$, or more than $p$ distinct colors appear on $H$. 
In several respects, this family of parameters is more similar to weak coloring numbers than strong coloring numbers are. 
First, note that 
\[
  \chi(G) = \chi_1(G) \leq  \chi_2(G) \leq \cdots \leq \chi_{\infty}(G) = \td(G),
\]
thus both parameters tend to treedepth, while strong coloring numbers tend to treewidth: $\scol_{\infty}(G)=\tw(G)+1$~\cite{vdHetal17}. 
Second, the known bounds for $\chi_p(G)$ for simple treewidth are similar: It is shown in~\cite{DFMS21} that $\chi_p(G) \in \Oh(p^{k-1}\log p)$ for graphs of simple treewidth at most $k$, and this is tight. 
For planar graphs, the best known upper bound is $O(p^3 \log p)$, while  stacked triangulations achieve $\Omega(p^2 \log p)$~\cite{DFMS21}. 
Thus the gap is almost the same as for weak coloring numbers. 
Interestingly, the approach based on chordal partitions, on which is based the $\Oh(r^3)$ bound for weak coloring numbers~\cite{vdHetal17}, does not seem to work for bounding $\chi_p(G)$. 
Another common trait is that, again, having bounded $p$-centered chromatic number for all $p$ coincide with the notion of bounded expansion, and like for weak coloring numbers, there are classes with polynomial expansion and super-polynomial $p$-centered chromatic numbers~\cite{DJPPP20}. 

Finally, a third related family of parameters are {\em fractional $\td$-fragility rates}, defined as follows. 
Given a positive integer $a$ and a graph $G$, consider the smallest integer $r(G, a)$ such that there exists a probability distribution on the vertex subsets of $G$ with the following two properties: (1) Each vertex $v\in V(G)$ has probability at most $1/a$ of belonging to a random subset sampled from this distribution, and (2)  $\td(G-X) \leq r(G, a)$ for each subset $X$ in the support of the distribution. 
A class of graphs is {\em fractionally $\td$-fragile at rate $r$} if $r(G, a) \leq r(a)$ for all graphs $G$ in the class and all positive integers $a$. 
This parameter was introduced by Dvo{\v{r}}{\'a}k and Sereni~\cite{DS20}. 
Among others they proved an upper bound of $r(a) = \Oh(a^3 \log a)$ for planar graphs, and a bound of $\Oh(a^2 \log a)$ for stacked triangulations which is tight. 
Again, there is a gap remaining for planar graphs, which is similar to that for the $p$-centered chromatic number and for weak coloring numbers. 
It is worth noting that, while all three families of parameters are genuinely different, there are similarities both in their behavior and in the proof ideas so far. 
It is thus natural to expect that if one could narrow the gap for planar graphs for one of these parameters, then this would likely lead to better bounds for the others as well.

\section*{Acknowledgments} 
We thank the two anonymous referees for their helpful comments, which improved the paper. We are particularly grateful to one referee for pointing out an error in an earlier version of the proof for the lower bound in Theorem 2. 
We also thank David Wood for discussions on this topic.


\providecommand{\noopsort}[1]{}

\end{document}